\documentclass[12pt]{amsart}

\usepackage{amsmath}
\usepackage{amssymb}
\usepackage{amsthm}
\usepackage[latin1]{inputenc}
\usepackage{eurosym}
\usepackage[dvips]{graphics}
\usepackage{graphicx}
\usepackage{epsfig}
\usepackage{hyperref}
\usepackage{dsfont}
\usepackage{caption}
\usepackage{subcaption}
\usepackage{epstopdf}
\usepackage{color}

\usepackage{ifthen}

\makeindex

\renewcommand{\div}{\operatorname{div}}

\newcommand{\Rr}{{\mathbb{R}}}

\newcommand{\Aa}{{\mathcal{A}}}

\newcommand{\Ff}{{\mathcal{F}}}

\newcommand{\bx}{{\bf x}}

\newcommand{\bdx}{\dot{\bf x}}
\newcommand{\bp}{{\varpi}}
\newcommand{\bdp}{\dot{\varpi}}

\newcommand{\bQ}{{\bf Q}}
\newcommand{\bdQ}{\dot{\bf Q}}

\newcommand{\Pp}{{\mathcal P}}

\newcommand{\lb}{\left(}
\newcommand{\rb}{\right)}

\usepackage[usenames,dvipsnames,svgnames,table]{xcolor}
\definecolor{darkgreen}{rgb}{0.,.66,0.}
\hypersetup{%
  breaklinks,
  bookmarks=true,
  bookmarksnumbered=true,
  bookmarksopen=true,
  colorlinks=true,
  linkcolor=blue!60!black,
  urlcolor=blue!60!black,
  citecolor=blue!60!black
}
\newboolean{include-notes}
\setboolean{include-notes}{true}
\newcommand{\js}[1]{\ifthenelse{\boolean{include-notes}}%
 {\textcolor{cyan}{\textbf{[ #1  --JS]}}}{}}
\newcommand{\dg}[1]{\ifthenelse{\boolean{include-notes}}%
 {\textcolor{darkgreen}{\textbf{[ #1  --DG]}}}{}}
\newcommand{\add}[1]{\ifthenelse{\boolean{include-notes}}%
 {\textcolor{blue}{#1}}{#1}}

\newtheorem{teo}{Theorem}

\newtheorem{pro}{Proposition}
\newtheorem{cor}{Corollary}
\newtheorem{hyp}{Assumption}
\newtheorem{rem}{Remark}
\newtheorem{prob}{Problem}

\theoremstyle{definition}

\begin{document}

\title[]{A mean-field game approach to price formation in electricity markets}

 \author{Diogo A. Gomes}
 \address[D. A. Gomes]{
         King Abdullah University of Science and Technology (KAUST), CEMSE Division, Thuwal 23955-6900. Saudi Arabia.}
 \email{diogo.gomes@kaust.edu.sa}
 \author{Jo\~ao Sa\'ude}
 \address[J. Saude]{
         Carnegie Mellon University, Electrical and Computer Engineering department. 5000 Forbes Avenue Pittsburgh, PA 15213-3890 USA.}
 \email{jsaude@andrew.cmu.edu}


\thanks{
        D. Gomes  was  partially supported by KAUST baseline funds and KAUST OSR-CRG2017-3452.
        J. Sa\'ude was partially supported by FCT/Portugal through the CMU-Portugal Program.
}
\date{\today}
\begin{abstract}
	Here, we introduce a price-formation model where a large number of small players can store and trade electricity. 
	Our model is a constrained mean-field game (MFG) where the price is a Lagrange multiplier for the supply vs. demand balance condition.  
	We establish the existence of a unique solution using a fixed-point argument. In particular, 
	we show that the price is well-defined and it is a Lipschitz function of time. Then, we study linear-quadratic models that can be
	solved explicitly and compare our model with real data.  
\end{abstract}
\maketitle
%
%
\section{Introduction} 
\label{sec:introduction}


	The mean-field game (MFG) framework \cite{Caines2, Caines1, ll1, ll2} models systems with many rational players (see, e.g.,  the surveys  \cite{GPV} and \cite{GS}). 
	Here, we are interested in the price formation in electricity markets. 
	In our model, a large number of agents owns storage devices that can be charged and later supply the grid with electricity. 
	Agents seek to maximize profit by trading electricity at a price $\bp(t)$, which is set by a supply versus demand balance condition.
	
	With the advent of electric cars, a large number of network-connected batteries are already available,
	and their number is only likely to increase. 
	Moreover, energy can be stored as heat or cold, using space or water heaters and air-conditioning units \cite{Kizilkale2014829, Kizilkale20141867,Kizilkale20143493}. 
	With new small network-capable devices, appliances can be connected to the grid and use smart algorithms to control their energy usage. 
	These algorithms can  balance supply and demand and, thus, are particularly relevant when combined with solar and wind energy production, where power demand seldom matches production.

	Price formation models were some of the first MFG models  \cite{ll3}. This line of research was pursued by several authors, see \cite{MR2835888, MR2817378, MR2573148, MR3210751, MR3078202, LachapelleLasryLehalleLions,  Gomes20164693} and the monograph \cite{GNP}.
	Some of these models are formulated as free boundary problems
	\cite{MR2835888, MR2817378}; others as a load control problem \cite{MR1029068, MR933047}. For example, 
	using mean-field control and MFG, the load-control problem through
	switching on and off
	space heaters 
	was studied 
	 in \cite{Kizilkale2014829, Kizilkale20141867,Kizilkale20143493}. 
	 Previous authors addressed the price issue by assuming that the demand is a given function of the price \cite{llg3}
	 or that the price is a given function of the demand, see
	 \cite{2011arXiv1110.1732C}, \cite{2017arXiv171008991C}, and \cite{de2016distributed}. In particular, in these references, the authors use a 
	 price function to study mean-field equilibrium in electricity markets in a setting that is
	 similar to ours.  
	
	 Here, we pursue a different approach: 
 	often, in economic models, prices of goods and services are determined by the balance between supply and demand rather than by a given function of the supply. Therefore, the price as a function of the supply or demand is not known a priori and a key unknown in the problem.  
	This observation motivated the approach in \cite{Gomes20164693}, where price arises from supply versus demand constraints. 
	However, that model is more complex than the one discussed here and was only studied from a numerical perspective. 
	Thus, mathematical issues such as
	the existence and uniqueness of a price, the well-posedness of the model, and the convergence of numerical methods
	were left unanswered and are settled here. 
		
	Our model comprises three quantities of interest: a price $\varpi\in C([0,T])$, a value function
	$u\in C(\Rr\times [0,T])$,  and a path describing the statistical distribution of the agents, $m\in C([0,T], \Pp)$, where $\Pp$ is the set of probability
	measures in $\Rr$ with bounded first moment, endowed with the $1$-Wasserstein distance. 
	These quantities are determined by the following problem.
	\begin{prob}
	\label{P1}
	Given $\epsilon\geq 0$, 
	a Hamiltonian, $H:\Rr\times\Rr\to \Rr$, $H\in C^\infty$, 
	an energy production rate $\bQ:[0,T]\to \Rr$, $\bQ\in C^\infty([0,T])$,
	a terminal cost ${\bar u}:\Rr\to \Rr$, ${\bar u}\in C^\infty(\Rr)$ and an initial probability distribution ${\bar m}\in \Pp\cap C^\infty_c(\Rr)$, 
	find $u:\Rr\times [0,T]\to \Rr$, $m\in C([0,T], \Pp)$,  and $\bp:[0,T]\to \Rr$ solving 
	\begin{equation}
	\label{mainsys}
	\begin{cases}
	-u_t + H(x,\bp(t)+u_x)=\epsilon  u_{xx}\\
	m_t - (D_p H(x,\bp(t)+u_x)m)_x=\epsilon m_{xx}\\
	\int_{\Omega} D_p H(x,\bp(t)+u_x) d m = -\bQ(t),
	\end{cases}
	\end{equation}	
	and satisfying the initial-terminal conditions
	\begin{equation} \label{itcond}
\begin{cases}
u(x,T)={\bar u}(x),\\
m(x,0)={\bar m}(x). 
\end{cases}
\end{equation}
\end{prob}
	
     In the previous problem, $x\in \Rr$ represents the 
	state of a typical agent; that is, the energy stored by the agent. 
	The function $u(x,t)$ is the value function for an agent whose charge is $x$ at time $t$. 
	The Hamiltonian, $H:\Rr\times\Rr\to \Rr$ is determined by the optimization problem that each agent seeks to solve, as described in Section \ref{sec:a_mean_field_model_for_price_formation}. 
	We require $u$ to be a viscosity of the first equation in \eqref{mainsys}. However, if $\epsilon>0$, parabolic regularity theory gives 
	additional regularity for $u$. 
	For each  $t\in [0,T]$, $m$ determines the distribution of the energy storage of the agents. Here, we assume that $m$ is a weak solution of the second equation in \eqref{mainsys}; that is, for every 
	$\psi\in C^2_c(\Rr\times [0, T])$, we have	
	\begin{align*}
	&\int_0^T\int_\Rr \left(\psi_t+\psi_x D_pH(x, \varpi +u_x)-\epsilon \psi_{xx}\right)m dx dt\\&\quad =
	\int_{\Rr} \psi(x,T) m(x,T)dx-\int_{\Rr} \psi(x,0) {\bar m}(x)dx. 
	\end{align*}
	The parameter $\epsilon$ corresponds to random fluctuation in the storage of the agents.   
	Finally, the spot price, $\bp(t)$, is selected so that the total energy used balances the supply, $\bQ(t)$, the condition imposed by the last equation in \eqref{mainsys}.

  In the current model, agents have a time horizon $T>0$, and, at time $T$, 
  they  incur in the terminal cost ${\bar u}(x)$ that depends on their state at the terminal time. For example, agents may prefer to have the batteries fully charged at the end of the day. 
  Moreover, the initial distribution of agents, ${\bar m}$, is known. These two facts are encoded in
   the initial-terminal boundary conditions, \eqref{itcond}. This model can easily be modified to address
   periodic in time boundary conditions and the infinite horizon discounted problem. 

	First, 
	in Section \ref{sec:a_mean_field_model_for_price_formation}, we present a derivation of our model
	and examine some of its mathematical properties. 
	Then, after a brief discussion of the main 
	assumptions, in Section \ref{AS}, we prove our main result
    given by the following theorem. 
	\begin{teo}
	\label{T1}
	Suppose that Assumptions \ref{A1}--\ref{A5} (see Section \ref{AS}) hold. Then, there exists a solution 
	$(u, m, \varpi)$ of Problem \ref{P1} where $u$ is a  viscosity solution of the first equation, Lipschitz and semiconcave in $x$, and differentiable almost everywhere with respect to $m$, $m\in C([0,T], \Pp)$, and $\varpi$ is Lipschitz continuous. 
	Moreover, if $\epsilon>0$ this solution is unique.  
	
	If $\epsilon=0$ and Assumption \ref{A6} holds, then there is a unique  solution $(u, m, \varpi)$. Moreover, $u$ is differentiable in $x$ for every $x$,  and  $u_{xx}$ and $m$ are bounded.  
	\end{teo}
	 
	\begin{rem}
	In the case $\epsilon>0$, the regularity of the solutions can be improved using parabolic regularity. 
	\end{rem} 
	 
	There are two main contributions of this paper. First, is the existence part of the preceding theorem  which is proved in 
	Section \ref{Es} using a fixed-point argument.
	The key step is establishing 
	an ordinary differential equation satisfied by the price, $\varpi$. Using this equation, we obtain Lipschitz bounds and then apply Schauder's fixed-point theorem. 
	To prove the uniqueness part of the theorem, 
	we use the monotonicity method. This is achieved in Section \ref{suniq} where we identify a new monotonicity structure for mean-field 
	games with constraints. Finally, we discuss
	linear-quadratic models, that can be solved explicitly and
	compare our model with the ones in \cite{de2016distributed}. Our results suggest that 
	a price determined by a supply versus demand condition may help stabilize the oscilations 
	of the price in particular in peak-demand situations.

\section{A mean-field model for price formation} 
\label{sec:a_mean_field_model_for_price_formation}
	Here, we present the derivation of our price model. To simplify the discussion, we examine the deterministic case, $\epsilon=0$. 
	We consider an electricity grid connecting consumers to producers of energy.
	In our model, each consumer has a storage device connected to the network, for example, an electric car battery. 
	We assume that  all devices are similar. 
	Consumers trade electricity, charging the batteries when the price is low and selling electricity to the market when the price is high.
	A typical consumer has a battery whose charge at time $t\in [0,T]$ is $\bx(t)$. 
	This charge changes according to an energy flow rate,
	the control variable selected by each consumer, 
	which is a bounded measurable function of $\alpha:[0,T]\to A$, where $A \subset \Rr$. 
	Positive values of $\alpha$ correspond to buying energy from the grid, and negative values to selling to the grid.  
	Accordingly, each consumer charge, $\bx$, changes according to the dynamics:
	\begin{equation*}
		\bdx(t) = \alpha(t). 
	\end{equation*}
	Each consumer seeks to select $\alpha$ to minimize its cost, thus maximizing profit. 
	This cost is determined by a terminal cost and by the integral of the \emph{running cost}, $\ell(\alpha,x,t)$, where $\alpha(t)$ is the energy traded with the electricity grid at time $t$,
	and $\ell$ depends in time through $\bp(t)$, the spot electricity price and is of the form
	\begin{equation} \label{lform}
		\ell(\alpha,x,t)=\ell_0(\alpha, x)+ \bp(t) \alpha(t). 
	\end{equation}
	In the preceding expression, the term $\bp(t) \alpha(t)$ is the instantaneous cost corresponding to a charging current $\alpha(t)$. 
	The current (or more precisely power), $\alpha$ is measured in Watt,  $\text{W}$, and the price, $\varpi$, in $\$ \text{W}^{-1}s^{-1}$.  
	The function $\ell_0$ accounts for non-linear effects of the current usage, for example, battery wear and tear, and for state preferences. 
	For example, we often take 
	\begin{equation} \label{quadl}
		\ell_0(\alpha,x,t) = \frac{c}{2} \alpha^2(t) +  V(x),
	\end{equation}   
    where $c$ is a constant that accounts for the battery's wear off, typically given in $\$ \text{W}^{-2} s^{-1}$, and $V(x)$ is a potential that takes into account battery constraints and charge preferences. 
	The singular case where
	\[
		V(x)=\begin{cases}
				0 \qquad \text{if}\ 0\leq x\leq 1\\
				+\infty\qquad \text{otherwise},
			 \end{cases}
	\]  
	corresponds to the case where the battery charges satisfies $0\leq x\leq 1$. 
	To avoid singularities, we work with smooth potentials growing as $x\to \pm \infty$; this behaviour correspond to a penalty on the battery charge rather than a hard constraint. 
	The nonlinear term, $\frac{c}{2}\alpha^2(t)$, models battery wear and tear, which is large in high-current regimes.  
	The particular quadratic form in \eqref{quadl} simplifies the mathematical treatment. 
	However, it can be replaced by a convex function of $\alpha$ without any major change in the discussion. 
	
	Each consumer minimizes the functional 
	\begin{equation}\label{J}
		J(x,t,\alpha)=\int_t^T \ell(\alpha(s), \bx(t),t) ds + {\bar u}(\bx(T)),
	\end{equation}
	where ${\bar u}$ is the \emph{terminal cost} and $\alpha\in \Aa_t$, where $\Aa_t$ is the set of bounded measurable functions $\alpha:[t,T]\to A$.
	
	The \emph{value function}, $u$, is the infimum of $J$ over all controls in $\Aa_t$; that is,
	\begin{equation*}
		u(x,t) = \inf_{\alpha \in \Aa_t} J(x,t,\alpha).
	\end{equation*}
    The \emph{Hamiltonian}, $H$, for the preceding control problem is
	\begin{equation*}
		H(x,p) = \sup_{a\in A} \lb - p a -\ell_0(x, a) \rb.
	\end{equation*}
	For example, for $\ell_0$ as in \eqref{quadl}, we have
	\[
		H(x, p)=\frac{p^2}{2c}+V(x).
	\]

	From standard optimal control theory, $u$ is a viscosity solution (see \cite{Bardi}) of the \emph{Hamilton-Jacobi equation}
	\begin{equation}
	\label{hjd}
	\begin{cases}
		-u_t + H(x,\bp(t)+ u_x)=0  \\
		u(x,T)={\bar u}(x).
	\end{cases}	
	\end{equation}
	For $\ell_0$ as in \eqref{quadl}, the prior equation becomes
	\begin{equation*}
	-u_t+\frac{1}{2c} (u_x+\bp(t))^2 - V(x)=0.
	\end{equation*}	
	
	Finally, at points of differentiability of $u$, the optimal control is given by
	\begin{equation*}
		\alpha^*(t) = -D_pH(x, \bp(t)+u_x(\bx(t),t)).
	\end{equation*}

	The associated \emph{transport equation} is the adjoint of the linearized Hamilton-Jacobi equation:
	\begin{equation}
	\label{te}
		\begin{cases}
		m_t - \lb D_p H(x, u_x+\bp(t))m \rb_x=0, 
		\\
		m(x,0)={\bar m}(x),
		\end{cases}
	\end{equation}
	where ${\bar m}$ is the initial distribution of the agents. 
	
	Taking $\ell_0$ as in \eqref{quadl}, the transport equation above becomes
	\begin{equation*}
		m_t - \frac{1}{c}(m(\varpi +u_x))_x =0.
	\end{equation*}
	
	Finally, we fix an \emph{energy production function} $\bQ(t)$ and require that the production balances demand. 
	Mathematically, this constraint corresponds to the identity
	\[
		\int_{\Rr} \alpha^*(t) m(x,t) dx = \bQ(t);
	\]
	that is, 
	\begin{equation}
	\label{bc}
	\int_{\Rr} D_p H(x, u_x+\bp(t))  m(x,t) dx = -\bQ(t).
	\end{equation}
	This foregoing equality is the balance equation that forces the consumed energy to match the  production; 
	this constraint determines the price, $\bp(t)$.
	
	
	%
	
	Combining \eqref{hjd}, \eqref{te} and \eqref{bc},  we obtain \eqref{mainsys} with $\epsilon=0$
	and the initial-terminal conditions \eqref{itcond}. 
		
	Now, we consider the case where the agents are subject to independent random consumption.  In this case $\epsilon>0$.
	Let $(\Omega, \Ff, P)$ be a probability space, where $\Omega$ is a \emph{sample space}, $\Ff$ a $\sigma$-\emph{algebra} on $\Omega$ and $P$ a \emph{probability measure}.
	Let $W_t$ be a Brownian motion on $\Omega$ and $\{\Ff_t\}_{t\geq 0}$ the associated \emph{filtration}. 
	In this case, we model the agent's motion by the stochastic differential equation 
	\begin{equation*}
		d\bx(t) = \alpha(t)dt+\sqrt{2\epsilon}d W_t,  
	\end{equation*}
	where the control, $\alpha$, is a bounded  progressively measurable real-valued process. 
	Following the previous steps and using standard arguments in stochastic optimal control, we arrive again at \eqref{mainsys}.

\section{Main Assumptions}
\label{AS}

We begin by discussing our main assumptions.  First, 
we suppose that $H$ is the Legendre
transform of a Lagrangian that is the sum of an 
``energy flow cost", $\ell_0(\alpha)$, 
and a ``charge preference cost", $V(x)$, as follows:
\begin{hyp}
	\label{A1}
	The Hamiltonian $H$ is the Legendre transform of a convex Lagrangian:
	\begin{equation}
	\label{ham}
	H(x, p)=\sup_{\alpha\in \Rr}  -p\alpha -\ell_0(\alpha)-V(x) , 
	\end{equation}
	where $\ell_0\in C^2(\Rr)$ is a uniformly convex function  and $V\in C^2(\Rr)$ is bounded from below. 
\end{hyp}

\begin{rem}
	\label{R1}
	The preceding hypothesis implies that the map 
	$p\mapsto H(x,p)$ is (strictly) convex.  Moreover,
	the Hamiltonian in \eqref{ham} can be written as 
	\begin{equation}
	\label{sf}
	H(x,p)=H_0(p)-V(x).
	\end{equation}
	Thus, 
	\[
	D^2_{xp}H(x,p)=0
	\]
	for all $x, p\in \Rr$. 
\end{rem}

To obtain a fixed point, we need several 
a priori estimates. These depend on convexity and regularity properties of the data. The following two 
assumptions lay out our requirements on the potential, $V$. 

\begin{hyp}
	\label{A2}
	The potential $V$ in \eqref{ham} and the terminal data ${\bar u}$ are globally Lipschitz.
\end{hyp}
\begin{hyp}
	\label{A3}
	The potential $V$ in \eqref{ham} and  the terminal data ${\bar u}$ 
	satisfy
	\[
	|D^2_{xx}V| \leq C, \qquad |D^2_{xx}{\bar u}|\leq C
	\]
	for some positive constant $C$. 
\end{hyp}

Next, we state an additional regularity for the initial-terminal data that is used to 
prove second-order estimates. 
\begin{hyp}
\label{A4}
There exists a constant, $C>0$, such that 
\[
|{\bar m}_{xx}|, |{\bar u}_{xx}|\leq C. 
\]
\end{hyp}

The next two assumptions are used to 
ensure the solvability of the demand-supply relation; that is, given $\bQ$ that we can determine a suitable price.

\begin{hyp}
	\label{A5}
	There exists $\theta>0$ such that 
	\[
	D^2_{pp} H(x,p) >\theta
	\]
	for all $x, p\in \Rr$. In addition, 
	there exists $C>0$ such that
	\[
	|D^{3}_{ppp}H|\leq C. 
	\]
\end{hyp}

\begin{rem}
	\label{R2}	
	Using \eqref{sf} in Remark \ref{R1}, the preceding assumption combined with Assumption \ref{A1}
	implies that the function $p\mapsto D_pH(p,x)$ is strictly increasing and 
	\[
	\lim_{p\to -\infty} D_pH(p,x)=-\infty\qquad \lim_{p\to +\infty} D_pH(p,x)=+\infty,
	\]
	uniformly in $x$. 
\end{rem}
\begin{rem}
	\label{R3}
	The uniform convexity of $\ell_0$ in Assumption \ref{A1} gives an upper bound for $D^2_{pp} H$. Thus, 
	Assumption \ref{A1}  and \ref{A5} imply
	\[
	|D^2_{pp}H(x,p)|\leq C
	\]
	for all $x, p\in \Rr$.
\end{rem}

The following hypothesis gives regularity and uniqueness of solutions in the first-order case.

\begin{hyp}
\label{A6}
The potential, $V$, and the terminal cost, ${\bar u}$, are convex. 
\end{hyp}

\section{Existence of a solution}
\label{Es}

Here, we establish the existence of a solution for the price model, 
\eqref{mainsys}, using a fixed-point argument on $\varpi$. 
In the following two propositions, we examine the Hamilton-Jacobi equation
\begin{equation}
\label{hj}
\begin{cases}
-u_t+H(x,\varpi+u_x)=\epsilon u_{xx}\\
u(x,T)={\bar u}(x).
\end{cases}
\end{equation}
First, using Assumption \ref{A2}, we prove the Lipschitz continuity of $u$. Next, using Assumption \ref{A3}, 
we obtain the semiconcavity of $u$. The proofs follow standard arguments in optimal control theory. However, we present them here to make it evident that the Lipschitz and semiconcavity constants are
uniform in $\varpi$ and $\epsilon$, both essential points in our argument. 

\begin{pro}
\label{P4}
Consider the setting of Problem \ref{P1} and suppose that Assumptions
\ref{A1} and \ref{A2} hold. Let 
$u$ solve \eqref{hj}. Then,
$u(x,t)$ is locally bounded and the map
  $x\mapsto u(x,t)$ is Lipschitz for $0\leq t\leq T$. Moreover, the Lipschitz bound on $u$ does not depend on $\varpi$
  nor on $\epsilon$. 
\end{pro}
\begin{proof}
The proof follows from the representation of  $u$ as a solution to a stochastic control problem (or deterministic if $\epsilon=0$).  
We fix a filtered probability space $(\Omega, \Ff_t, P)$ that supports a one-dimensional Brownian motion $W_t$. Then,
\[
u(x,t)=\inf E\left[\int_t^T \ell_0(\alpha)+\varpi \alpha+V(\bx) ds +{\bar u}(\bx(T))\right],
\]
where the infimum is taken over bounded progressively measurable
controls $\alpha:[t, T]\to \Rr$ and $\bx$ solves the stochastic differential equation
\[
d\bx =\alpha dt +\sqrt{2\epsilon} dW_t. 
\] 
To prove local boundedness, we use the sub-optimal control $\alpha\equiv 0$ to get an upper bound, 
and the fact that $V$ is bounded by below to obtain the lower bound. We observe, however, 
that the lower bound depends on bounds on $\varpi$. 

Then, we fix an optimal control, $\alpha^*$, for $(x,t)$; that is, 
\[
u(x,t)= E\left[\int_t^T \ell_0(\alpha^*)+\varpi \alpha^*+V(\bx^*) ds +{\bar u}(\bx(T)^*)\right]. 
\]
Then, for any $h\in \Rr$, we have
\[
u(x+h, t)\leq E\left[\int_t^T \ell_0(\alpha^*)+\varpi \alpha^*+V(\bx^*+h) ds +{\bar u}(\bx(T)^*+h)\right],
\]
from which the Lipschitz bound follows. Note that this Lipschitz bound does not depend
on $\varpi$, only on $T$ and on the Lipschitz estimates for $V$ and ${\bar u}$. 
\end{proof}

\begin{pro}
\label{P2}
Consider the setting of Problem \ref{P1} and suppose that Assumptions
\ref{A1} and \ref{A3} hold.
	Then, $x\mapsto u(x,t)$ is semiconcave with a semiconcavity constant that does not 
	depend on $\epsilon$ nor on $\varpi$. 
\end{pro}
\begin{proof}
As before, we fix an optimal control $\alpha^*$
 for $(x,t)$; that is, 
\[
u(x,t)= E\left[\int_t^T \ell_0(\alpha^*)+\varpi \alpha^*+V(\bx^*) ds +{\bar u}(\bx(T)^*)\right]. 
\]
Then, for any $h\in \Rr$, we have
\[
u(x\pm h, t)\leq E\left[\int_t^T \ell_0(\alpha^*)+\varpi \alpha^*+V(\bx^*\pm h) ds +{\bar u}(\bx(T)^*\pm h)\right]. 
\]
Therefore,
\[
u(x+h,t)-2 u(x,t)+u(x-h,t)\leq C h^2. 
\]
Note that $C$ does not depend
on $\varpi$, only on $T$ and on the semiconcavity estimates for $V$ and ${\bar u}$. 
\end{proof}

We have the following stability properties for the solutions 
of \eqref{hj}. 

\begin{pro}
\label{P6}
Consider the setting of Problem \ref{P1} 
and suppose  that Assumptions
\ref{A1}--\ref{A3} hold. Suppose that 
$\varpi_n\to \varpi$ uniformly on $[0,T]$, then $u^n\to u$ locally uniformly and $u^n_x\to u_x$ almost everywhere. 
\end{pro}
\begin{proof}
The local uniform convergence of $u^n$ follows from the stability of viscosity solutions. 

Because $u^n$ is semiconcave and converges uniformly to $u$, $u^n_x\to u_x$ almost everywhere. 
\end{proof}

Now, we examine the Fokker-Planck equation. 
\begin{equation}
\label{fp}
\begin{cases}
	m_t-\div(m D_pH(x, \varpi+u_x))=\epsilon\Delta m,\\
	m(x,0) = \bar m(x).
\end{cases}
\end{equation}
Let $\Pp$ denote the set of probability measures on $\Rr$ with finite second-moment and endowed
with the  $1$-Wasserstein distance. 
\begin{pro}
\label{P7}	
Consider the setting of Problem \ref{P1}  with 
$\epsilon> 0$ and suppose that Assumptions
\ref{A1}--\ref{A2} hold.
Then, \eqref{fp} has a solution $m\in C([0,T], \Pp)$. Moreover, 
\begin{equation}
\label{d1c}
d_1(m(t), m(t+h))\leq C h^{1/2}. 
\end{equation}

In addition, if Assumptions \ref{A3} and \ref{A4} hold, for any sequence 
$\varpi_n\to \varpi$ uniformly on $[0,T]$ and corresponing solutions $u_n$ of 
\eqref{hj} and $m_n$ of \eqref{fp}, 
we have
$m_n\to m$ in $C([0,T], \Pp_1)$.
\end{pro}
\begin{proof}
The existence of a solution in $C([0,T], \Pp_1)$ and the estimate in \eqref{d1c} were proven in 
\cite{cardaliaguet}. We note that, 
for $\epsilon\leq \epsilon_0$, 
 the constant $C$ can be chosen to depend only on $\epsilon_0$, on the problem data, and on $\|\varpi\|_{L^\infty}$. 
Thus, by the Ascoli-Arzela theorem, we have that $m_n\to m$ in $C([0,T], \Pp_1)$ for some $m\in C([0,T], \Pp_1)$. Because $\epsilon>0$, $m_n\to m$ in, for example, 
$L^2(\Rr\times [0,T])$. Moreover, 
\eqref{fp} has a unique solution. 
Thus,  it suffices to check that $m$ solves  \eqref{fp}. Because $u^n_x\to u_x$, almost everywhere, by semiconcavity, we have for any $\psi\in C^\infty_c(\Rr\times [0, T])$
\[
\int_0^T \int_{\Rr} \psi_x D_pH(x, \varpi^n+u_x^n) m^n dx dt\to 
\int_0^T \int_{\Rr} \psi_x D_pH(x, \varpi+u_x) m dx dt, 
\]
which gives that $m$ is a weak solution of \eqref{fp}. 
\end{proof}

Next, we prove an estimate for solutions of the system comprising \eqref{hj} and \eqref{fp}.
\begin{pro}
\label{P8}
Consider the setting of Problem \ref{P1}  with 
$\epsilon> 0$ and suppose that Assumptions \ref{A1} and \ref{A4} hold. 
Let $(u,m)$ solve  \ref{hj} and \ref{fp}.
Then
\begin{equation}
\label{soe}
\int_0^T \int_{\Rr} D^2_{pp}H u_{xx}^2 m dx dt \leq C
\end{equation}
\end{pro}
\begin{proof}
We begin by differentiating  \eqref{hj} twice with respect to $x$, multiply by $m$, and integrate by parts using \eqref{fp}. 
\end{proof}

\begin{rem}
Formally, the previous estimates hold for $\epsilon=0$. However, the above proof requires that $u$ is three times differentiable, which is not usually the case. Nevertheless, the estimate in \eqref{soe} 
is uniform in $\epsilon$. 
\end{rem}


Finally, we consider the price-supply relation. 
Due to Remark \ref{R2} and to the Lipschitz continuity of $u$ given by Proposition \ref{P4}, 
there exists a unique $\vartheta_0$ such that 
\begin{equation}
\label{id}
\int_{\Rr} D_p H(x, \vartheta_0+u_x(x,0) ){\bar m}dx=-\bQ(0).
\end{equation}
Moreover, $\vartheta_0$ is bounded by a constant that depends only on the problem data. 

Next, we differentiate 
\[
\int_{\Rr} D_pH(x, \varpi +u_x) m dx=-\bQ(t)
\]
in time to get the identity
\begin{equation}
\label{ddsc}
\dot \varpi \int_{\Rr} D^2_{pp} H m dx
+
\int_{\Rr} \left[ D^2_{pp} Hu_{xt}m+ D_p H m_t \right]dx=-\bdQ. 
\end{equation}
Differentiating \eqref{hj} in $x$ and substituting \eqref{fp} both quantities on the second term of the left hand side of \eqref{ddsc}, we get the following identity
\begin{align*}
\int_{\Rr} D^2_{pp} Hu_{xt}m+ D_p H m_t=&
\int_{\Rr} D^2_{pp} H \left(-\epsilon \Delta u_x+D_pH u_{xx}+ D_xH\right) m\\&+
\int_{\Rr} D_p H \left(\epsilon \Delta m+(m D_pH)_x\right).
\end{align*}
If Assumption \ref{A1} holds, we have by Remark \ref{R1} that $D^2_{xp}H=0$. Hence, 
\begin{equation}
\label{dynp}
\int_{\Rr} D^2_{pp} Hu_{xt}m+ D_p H m_t=
\int_{\Rr} D^2_{pp} H D_xH m+\epsilon D^{3}_{ppp}H u_{xx}^2 m. 
\end{equation}

Accordingly, we have the identity
\begin{equation}
\label{pode}
\dot \varpi \int_{\Rr} D^2_{pp} H m
=-\bdQ-\int_{\Rr} \lb D^2_{pp} H D_xH +\epsilon D^{3}_{ppp}H u_{xx}^2 \rb m. 
\end{equation}
Thus, given $\varpi$, we solve  \eqref{hj} and \eqref{fp} and define the following 
ordinary differential equation
\begin{equation}
\label{node}
\begin{cases}
\dot \vartheta=\frac{-\bdQ-\int_{\Rr} D^2_{pp} H(x, \varpi +u_x)  D_xH(x, \varpi +u_x)m  +\epsilon D^{3}_{ppp}H(x, \varpi +u_x)   u_{xx}^2 m }{ \int_{\Rr} D^2_{pp} H(x, \varpi +u_x) m}\\
\vartheta(0)=\vartheta_0,
\end{cases}
\end{equation}
where $\vartheta_0$ is determined by \eqref{id}.
Then, $(u, m, \varpi)$ solves \eqref{mainsys} if 
$\varpi$ solves \eqref{node}. 

\begin{pro}
\label{P9}
Consider the setting of Problem \ref{P1}  with 
$\epsilon> 0$
and suppose that Assumptions \ref{A1}--\ref{A5} hold. 
Suppose that $\varpi^n\to \varpi$ uniformly in $C([0,T])$. Let $u^n$,  $m^n$, and $\vartheta^n$ be the solutions to \eqref{hj},  \eqref{fp}, and
\ref{node} with $\varpi$ replaced by $\varpi^n$. 
Then, $\vartheta^n$ converges to $\vartheta$, uniformly in $C([0,T])$, where $\vartheta$ solves \eqref{node}.
Moreover, there exists a constant $C$ that depends only on the problem data but not on $\varpi$ such that 
$\|\vartheta\|_{W^{1,\infty}([0,T])}\leq C$.
\end{pro}
\begin{proof}
The bound in $W^{1,\infty}([0,T])$ for $\vartheta$ is a consequence of Remark \ref{R2} and of
the bounds in Assumption \ref{A5}, in Remark \ref{R3}, and in Proposition \ref{P8}. 
	
According to Proposition \ref{P6}, the uniform convergence of  $\varpi_n\to \varpi$ gives the convergence 
of $u^n_x\to u_x$, almost everywhere. In addition, Proposition \ref{P7} gives the convergence $m_n\to m$ in 
$C([0,T], \Pp)$. Because $D^2_{pp}H$ is bounded from below by Assumption \ref{A5}, we have the convergence
of the right-hand side of \eqref{node} as follows, for any $\psi\in C([0,T])$, 
\[
\int_0^T \psi \dot \vartheta_n ds\to \int_0^T \psi\dot \vartheta ds. 
\]
Also, because the family $\vartheta_n$ is equicontinuous, any subsequence has a further convergent 
subsequence that must converge to $\vartheta$. Thus, $\vartheta^n\to \vartheta$, uniformly. 
\end{proof}

With the preceding estimates, we can now prove a fixed-point result and show the existence of a solution for $\epsilon>0$. 

\begin{proof}[Proof of Theorem \ref{T1} - part 1, existence for $\epsilon\geq 0$]
	
We begin by addressing the case $\epsilon>0$. 	
According to Proposition \ref{P9}, the map $\varpi \to \vartheta$ determined by 
\eqref{hj},  \eqref{fp}, and
\eqref{node} is continuous in $C([0,T])$, bounded,  and compact due to the $W^{1, \infty}$ bound for $\varpi$. 
Thus, by Schauder's fixed-point theorem, it has a fixed point. 

Now, we examine the case $\epsilon=0$. The key 
difficulty is the continuity of the map $\varpi\to m$ in the case $\epsilon=0$. 
To overcome this difficulty, we use 
 the vanishing viscosity method and the techniques in \cite{E3}.

Let $(u^\epsilon, m^\epsilon, \varpi^\epsilon)$ solve \eqref{mainsys} with $\epsilon>0$. 
By the above, we have that $\varpi^\epsilon$ is uniformly bounded. 
Moreover, by Proposition \ref{P4},  $u^\epsilon$ is uniformly locally bounded and Lipschitz. Therefore, 
as $\epsilon\to 0$,  extracting a subsequence if necessary, 
$\varpi^\epsilon \to \varpi$ and $u^\epsilon\to u$ where $u$ is a viscosity solution of \eqref{hj}. 

Now, we introduce a phase-space measure $\mu^\epsilon$ as follows
\[
\int_0^T \int_{\Rr^2} \psi(x, p, t) d\mu^\epsilon(x,p,t)=
\int_0^T \int_{\Rr} \psi(x, \varpi^\epsilon+u_x^\epsilon, t) m^\epsilon dx dt
\]
for all $\psi\in C_b(\Rr\times \Rr \times [0,T])$.  Because $m^\epsilon\in C([0,T], \Rr)$ with a modulus of continuity that is uniform in $\epsilon$, as $\epsilon\to 0$,  we have $\mu^\epsilon \rightharpoonup \mu$; 
that is 
\[
\int_0^T \int_{\Rr^2} \psi d\mu^\epsilon \to \int_0^T \int_{\Rr^2} \psi d\mu. 
\]
Moreover,{\color{red} due to the strict convexity of the Hamiltonian}, arguing as in \cite{E3}, we have
\begin{align*}
&\int_0^T \int_{\Rr^2} 
\psi_t-D_pH(x,p)D_x\psi d\mu\\&\quad =\int_{\Rr} \psi(x,T) m(x,T)dx-\int_{\Rr} \psi(x,0) {\bar m}(x)dx. 
\end{align*}

Next, we fix $\delta>0$ and consider a standard mollifier $\eta_\delta$. We
define
\[
v^\delta=\eta_\delta*u. 
\]
We note that $|D^2 v^\delta|\leq \frac C {\delta^2}$. 
Then, using the uniform convexity of the Hamiltonian, we get
{\color{red}
\[
-v^\delta_t+\eta_\delta*|u_x-v^\delta_x|^2 + H(x, \varpi+v^\delta_x)\leq O(\delta). 
\]
}
Therefore, $w=v^\delta-u^\epsilon$ satisfies
\begin{align*}
&-w_t+D_p H (x, \varpi^\epsilon+u_x^\epsilon) w_x-\epsilon w_{xx}\\
&+\eta_\delta*|u_x-v^\delta_x|^2 +\gamma |\varpi+v^\delta_x-\varpi^\epsilon-u_x^\epsilon|^2
\leq O(\delta)+O(\frac{\epsilon}{\delta^2}). 
\end{align*}
Integrating with respect to $m^\epsilon$, we conclude that 
\[
\int_0^T\int_{\Rr^2}
\eta_\delta*|u_x-v^\delta_x|^2 +\gamma |\varpi+v^\delta_x-p|^2 d\mu^\epsilon
\leq O(\delta)+O(\frac{\epsilon}{\delta^2})+\|v^\delta-u^\epsilon\|_{L^\infty}.
\]
Next, we let $\epsilon\to 0$, to get
\[
\int_0^T\int_{\Rr^2}
\eta_\delta*|u_x-v^\delta_x|^2 +\gamma |\varpi+v^\delta_x-p|^2 d\mu
\leq O(\delta).
\]
Finally, by letting $\delta\to 0$, we conclude that $m$-almost every point is a point of approximate continuity of $u_x$. Therefore, $v^\delta_x\to u_x$ almost everywhere. 
Hence, $p=\varpi+u_x$ $\mu$-almost everywhere.  Therefore,
we obtain 
\begin{align*}
&\int_0^T \int_{\Rr^2} 
\left(\psi_t-D_pH(x,\varpi+u_x)D_x\psi \right)d\mu\\&\quad =\int_0^T \int_{\Rr} 
\left(\psi_t-D_pH(x,\varpi+u_x)D_x\psi \right)m dx dt\\&\quad  =\int_{\Rr} \psi(x,T) m(x,T)dx-\int_{\Rr} \psi(x,0) {\bar m}(x)dx,  
\end{align*}
which gives that $m$ solves \eqref{fp} with $\epsilon=0$. 

Note also, that the preceding reasoning implies that $u$ is differentiable almost everywhere with respect to $m$. 
\end{proof}

Finally, we record two additional results for \eqref{mainsys}. The first is an energy estimate that 
is similar to other results in MFG. 

\begin{pro}
Let $(u, m , \varpi)$ be the solution of Problem \ref{P1} constructed in Theorem \ref{T1}. Suppose that
there exists $C>0$ such that 
\[
p D_p H(x, p)-H(x,p)\geq \frac {1}{C} H(x, p)-C.
\]	
Then,	
\[
\int_0^T \int_{\Rr} H(x,\varpi+u_x) (m_0+m) dx dt \leq C.
\]
\end{pro}
\begin{proof}
We take the first equation in \eqref{mainsys} and multiply it by $\bar m-m$, and the second equation 
by
$u-\bar u$. Adding the resulting expressions and integrating by parts results in the desired estimate.
\end{proof}

The last result in this section concerns the regularity of the solutions \eqref{hj} in the 
case where both the  potential and terminal data are convex.  

\begin{pro}
\label{PZ1}
Suppose that $\epsilon=0$,  that Assumptions \ref{A1},  \ref{A3}, and \ref{A6} hold and let $\varpi$ be a Lipschitz function. 
Then, the solution to \eqref{hj}
is differentiable in $x$ for every $x\in \Rr$. Moreover,  $u_{xx}$ is bounded. 
\end{pro}
\begin{proof}
Due to Assumption \ref{A6}, we see that $u(x,t)$ is convex in $x$ by direct inspection of
the variational problem  \eqref{J}. By Proposition \ref{P2}, $u$ is semiconcave in $x$. This
gives the bound for $u_{xx}$ and the differentiability of $u$ in $x$. 
\end{proof}	
%


The preceding proposition implies the regularity of the solutions of Problem \ref{P1}, as stated in the next Corollary. 
\begin{cor}
\label{C1}
 Suppose that Assumptions \ref{A1}--\ref{A6} hold
 and that $\epsilon=0$. 
 Then, there exists a solution $(u, m, \varpi)$ of Problem \ref{P1} 
with $u$ differentiable in $x$ for every $x$ and $u_{xx}$ bounded. Moreover, $m$ is also bounded.  
\end{cor}
\begin{proof}
The result follows
by combining Proposition \ref{PZ1} with 
the fact that the transport equation with locally Lipschitz coefficients
has a unique weak solution in $L^\infty$.  
\end{proof}

\begin{proof}[Proof of Theorem \ref{T1} - part 2, additional regularity for $\epsilon=0$]
Additional regularity for the case where Assumption \ref{A6} holds and $\epsilon=0$ follows from Corollary \ref{C1}.
\end{proof}

\section{Uniqueness}
\label{suniq}

Now, we examine the uniqueness of solutions. We begin by observing that \eqref{mainsys} 
can be written as a monotone operator. As a consequence, we obtain a uniqueness result. 


We set
\[\Omega_T=\Rr\times [0,T],
\]
and
\[
\begin{split}
D=& (C^\infty(\Omega_T)\cap C([0,T], \Pp))\times (C^\infty(\Omega_T)\cap W^{1, \infty}(\Omega_T)) \times C^\infty([0,T]),\\
D_{+}=&\{ (m,u,\varpi) \in D~\mbox{s.t.}~m>0   \},\\
D^b=&\{ (m,u,\varpi) \in D~\mbox{s.t.}~m(x,0)={\bar m}(x),~u(x,T)={\bar u}(x)\},\\
D^b_{+}=&D^b \cap D_{+},\\
\end{split}
\]
Then, we define $A:D^b_{+} \to D$ as
\begin{equation}\label{eq:A}
\begin{split}
A \begin{bmatrix}
m\\
u\\
\varpi
\end{bmatrix}=& A_1 \begin{bmatrix}
m\\
u\\
\varpi
\end{bmatrix}+A_2 \begin{bmatrix}
m\\
u\\
\varpi
\end{bmatrix}\\
=&\begin{bmatrix}
u_t+\epsilon u_{xx}\\
m_t - \epsilon m_{xx}\\
0
\end{bmatrix}
+
\begin{bmatrix}
-H(x,Du+\varpi)\\
-\mathrm{div}(m D_p H(x,\varpi+u_x) )\\
\int\limits_{\Omega} m D_p H(x,\varpi+u_x)dx+\bQ(t)
\end{bmatrix}.
\end{split}
\end{equation}
Furthermore, for $w=(m,u,\varpi),~\tilde{w}=(\tilde{m},\tilde{u},\tilde{\varpi}) \in D$, we set
\[
\langle w,\tilde{w} \rangle = \int\limits_{\Omega_T} \lb m \tilde{m}+u\tilde{u} \rb dxdt + \int\limits_0^T \varpi \tilde{\varpi} dt.
\]
Then, $A$ is a {\em monotone operator} if
\begin{equation*}
\langle A[w]-A[\tilde w],w-\tilde w \rangle \geq 0\quad \mbox{for all}\quad w,\tilde{w} \in D^b_{+}.
\end{equation*}

Under the convexity of the map $p\mapsto H(x,p)$, $A$ is a monotone operator. 
\begin{pro}
\label{MP}
	Suppose the map $p\mapsto H(x,p)$ is convex. Then $A$ is a monotone operator.
\end{pro}
\begin{proof}
	Let $w=(m,u,\varpi),~\tilde{w}=(\tilde{m},\tilde{u},\tilde{\varpi}) \in D^b_{+}$. Then, 
	integrating by parts, we obtain
	\begin{equation*}
	\begin{split}
	&\langle A_1[w]-A_1[\tilde w],w-\tilde w \rangle\\
	=&\int\limits_{\Omega_T} ((u-\tilde{u})_t+\epsilon \Delta (u-\tilde u))  (m-\tilde{m})+\int\limits_{\Omega_T} ((m-\tilde{m})_t-\epsilon \Delta (m-\tilde m))  (u-\tilde{u})\\
	=&0,
	\end{split}
	\end{equation*}
	because $u-\tilde u$ and $m-\tilde m$ vanish at $t=0$ and $t=T$. Furthermore, we have that
	\begin{equation*}
	\begin{split}
	&\langle A_2[w]-A_2[\tilde w],w-\tilde w \rangle\\
	=&-\int\limits_{\Omega_T} (H(x, u_x+\varpi)-H(x,\tilde{u}_x+\tilde{\varpi}))(m-\tilde m) dxdt\\
	&-\int\limits_{\Omega_T} \mathrm{div}( m D_p H(x,u_x+\varpi)-\tilde m D_p H(x,\tilde{u}_x+\tilde{\varpi}))(u-\tilde u) dxdt\\
	&+\int\limits_0^T (\varpi -\tilde{\varpi}) \int_{\Rr} (m D_p H(x,u_x+\varpi)-\tilde m D_p H(x,\tilde{u}_x+\tilde{\varpi})) dx dt\\
	=&-\int\limits_{\Omega_T} (H(x, u_x+\varpi)-H(x,\tilde{u}_x+\tilde{\varpi}))(m-\tilde m) dxdt\\
	&+\int\limits_{\Omega_T} ( m D_p H(x,u_x+\varpi)-\tilde m D_p H(x,\tilde{u}_x+\tilde{\varpi}))(u_x-\tilde u_x) dxdt\\
	&+\int\limits_{\Omega_T} (m D_p H(x,u_x+\varpi)-\tilde m D_p H(x,\tilde{u}_x+\tilde{\varpi})) (\varpi -\tilde{\varpi})dx dt\\
	=&\int\limits_{\Omega_T}m \bigg(H(x,\tilde u_x+\tilde{\varpi})-H(x,u_x +\varpi)-( \tilde u_x +\tilde \varpi-u_x -\varpi) D_pH(x,u_x+\varpi)\bigg)dxdt\\
	&+\int\limits_{\Omega_T}\tilde m \bigg(H(x,u_x+\varpi)-H(x, \tilde u_x +\tilde \varpi)-( u_x + \varpi- \tilde u_x -\tilde \varpi) D_pH(x,\tilde u_x+\tilde \varpi)\bigg)dxdt\\
	\geq &0,
	\end{split}
	\end{equation*}
	by the convexity of $p\mapsto H(x,p)$. Combining the previous inequalities, we conclude that 
	\[
	\begin{split}
	&\langle A[w]-A[\tilde w], w- \tilde w \rangle \\
	=& \langle A_1[w]-A_1[\tilde w], w- \tilde w \rangle+\langle A_2[w]-A_2[\tilde w], w- \tilde w \rangle \geq 0.
	\end{split}
	\]
\end{proof}

Now, we discuss the last part of the proof of Theorem \ref{T1}. 

\begin{proof}[Proof of Theorem \ref{T1} - part 3, uniqueness]

Let $(m, u, \varpi)$ and $(\tilde m, \tilde u, \tilde \varpi)$ solve Problem \ref{P1}. If $\epsilon>0$
or if $\epsilon=0$ and Assumption \ref{A6} holds, 
we have $m$ and $\tilde m$ are absolutely continuous with respect to the Lebesgue measure.
Thus, 
the computations in the proof of Proposition \ref{MP}, combined with the uniform convexity of $H$
in Assumption \ref{A5}, give
\[
\int_0^T\int_{\Rr} |\varpi+u_x-\tilde \varpi -\tilde u_x|^2 (\tilde m+m)=0. 
\]
Therefore, $\varpi+u_x=\tilde \varpi +\tilde u_x$ almost everywhere. In both cases, this implies
\[
u_t=\tilde u_t, 
\]
almost everywhere and, thus, $u=\tilde u$. Finally, the uniqueness of the Fokker-Planck equation, for $\epsilon>0$ or for the transport equation, when $\epsilon=0$ and  Assumption \ref{A6} holds, give
$m=\tilde m$. 
\end{proof}

\section{Linear-quadratic models} 
\label{sec:linear_quadratic_models}
Here, we consider linear-quadratic price models. 
First, we examine the case without a potential and determine an explicit solution. 
Then, we introduce a quadratic potential that accounts for charge level preferences. In this last case, we describe a procedure to 
solve the problem, up to the inversion of Laplace transforms and solution of
ordinary differential equations. 
\subsection{State-independent quadratic cost} 
\label{sub:state_independent_quadratic_cost}
First, we consider the quadratic state-independent cost
\begin{equation} \label{eq:run_cost}
\ell(t,\alpha)= \frac{c}{2} \alpha^2 +\alpha\bp(t),
\end{equation}
where $c$ is a constant that accounts for the usage-depreciation of the battery.
The corresponding MFG is
\begin{align} \label{eq:MFG_quadratic}
\begin{cases}
-u_t+\frac{(\bp(t)+u_x)^2}{2c}=0\\
m_t-\frac{1}{c}(m(\bp(t)+u_x) )_x=0\\
\frac{1}{c} \int_{\Rr} (\bp(t)+u_x) m dx=-\bQ(t). 
\end{cases}
\end{align}
The stored energy by each agent follows optimal trajectories that solve the Euler Lagrange equation: 
\[
c \ddot{\bx}+\bdp=0. 
\]
Integrating the previous equation in time, we get 
\begin{equation} \label{opttraj}
\bdx(t)=\frac{1}{c} \lb \theta-\bp(t) \rb, 
\end{equation}
where $\theta$ is time independent. 
Next, by differentiating the Hamilton-Jacobi equation, we get
\[
-(u_x)_t+(u_x+\bp) \frac{u_{xx}}{c}=0. 
\]
Using the previous equation, taking into account the transport equation, and integrating by parts, we have
\begin{equation*}
\begin{split}
&	\frac{d}{dt} \int_{\Rr} u_x m dx = \int_{\Rr} u_{xt} m + u_x m_t = \int_{\Rr}  u_{xt} m + \frac 1 c u_x \lb m(\varpi+u_x) \rb_x \\
&= \frac 1 c\int_{\Rr} (\varpi+u_x) u_{xx} m - u_{xx} m (\varpi+u_x) dx= 0,
\end{split}
\end{equation*}
assuming that $m$ has fast enough decay at infinity. 

Thus, the supply vs demand balance condition becomes
\[
\bQ(t)=-\frac{1}{c}\int_{\Rr} (u_x+\bp) m dx = \frac{1}{c} \lb \Theta -\bp \rb, 
\]
where
\begin{equation}
\label{thetadef}
\Theta=-\int_{\Rr} u_x m dx
\end{equation}
is constant.
From the above, we obtain the following linear price-supply relation
\begin{equation}
\label{lpsr}
\bp =\Theta- c\bQ(t).
\end{equation}
Integrating \eqref{opttraj} in time and taking into account the linear price-supply relation \eqref{lpsr}, we gather
\begin{equation}
\bx(T) = \bx(t)+ \frac{1}{c} \int_t^T (\theta - \bp(s) )ds= x + \frac{T-t}{c} \lb \theta-\Theta \rb + \int_t^T  \bQ(s) ds.
\end{equation}
Accordingly, $u$ is given by the optimization problem
\begin{align*}
u(x,t)
=&\inf_{\theta} 
\int_t^T \left[\frac{(\theta-\Theta +c \bQ(s))^2}{2 c}
+\frac 1 c (\theta-\Theta +c \bQ(s)) (\Theta -c \bQ(s))\right] ds\\&
+\bar u\left(x+\frac{(\theta -\Theta)}{c} (T-t)+K(t)\right),
\end{align*}
where
\[
K(t)=\int_t^T Q(s)ds.
\]
By setting $\mu =\theta-\Theta$, we get 
\begin{align*}
u(x,t)
=&\inf_{\mu} 
\int_t^T \left[\frac{(\mu +c \bQ(s))^2}{2c}
+\frac 1 c (\mu +c\bQ(s)) (\Theta -c\bQ(s))\right] ds\\&
+\bar u\left(x+\frac \mu c (T-t)+K(t)\right).
\end{align*}
Thus, given $\Theta$, we determine a function, $u^\Theta$, solving the preceding minimization problem. For that, we expand the integral
 to get
\begin{equation*}
\begin{split}
u^\Theta(x,t) = \inf_{\mu}\Bigg[ \frac{T-t}{2 c} \mu^2 + \frac 1 c (T-t) \Theta \mu + \int_t^T \lb \Theta -c\frac{\bQ(s)}{2} \rb \bQ(s) ds \\
+ \bar u \lb x+\frac \mu c (T-t)+K(t) \rb\Bigg].
\end{split}
\end{equation*}
Next, we take the derivative of the right-hand side of the prior identity with respect to $\mu$ and obtain the relation
\begin{equation}
\label{impmu}
\mu +  \bar u_x\lb \bx(T) \rb = -\Theta. 
\end{equation}
If $\bar u$ is a convex function, the preceding equation has a unique solution, $\mu(\Theta)$ for each given $\Theta$. 
Thus, given $\Theta$, we obtain a solution, $u^\Theta$ for the Hamilton-Jacobi equation. 
Finally, we use the resulting expression for $u^\Theta$ in \eqref{thetadef} at $t=0$ to obtain the following condition for  $\Theta$:
\begin{equation}
\label{thetaimp}
\Theta=-\int_{\Rr} u_x^\Theta(x,0) m_0(x) dx. 	   
\end{equation}
Solving the preceding equation, we obtain $\Theta$ and hence $\varpi$ using the price-supply relation, \eqref{lpsr}. 

As an example, we consider the terminal cost
\begin{equation*}
\bar u(y) = \frac{\gamma }{2}\lb y-\zeta \rb^2. 
\end{equation*}

Solving\eqref{impmu}, we obtain
\begin{equation} \label{eq:mu}
\mu=-\frac{\gamma(K(t)+x-\zeta) + \Theta}{1+\gamma \frac{T-t}{c}}.
\end{equation}
Accordingly, we have
\begin{align*}
&u^\Theta(x,t)=\\&
\frac{\gamma (K(t)+x-\zeta)^2+ \frac{(t-T)}{c} \Theta (	
2\gamma (K(t)+x-\zeta) +\Theta)}{2 \left(1+\gamma \frac{T-t}{c}\right)}\\&+
\Theta K(t)
- c\int_t^T \frac{\bQ^2(s)}{2}ds.
\end{align*}
Therefore, 
\[
u_x(x,t)=
\gamma\frac{ K(t)+x-\zeta-\frac{(T-t)}{c} \Theta}{1+\gamma \frac{T-t}{c}}
\]
Using the previous expression for $t=0$ in 
 \eqref{thetaimp}, we obtain the following equation for $\Theta$
\[
\Theta=-
\gamma\frac{ K(0)+\bar x-\zeta-\frac{T}{c} \Theta}{1+\gamma \frac{T}{c}}
\]
where
\[
\bar x=\int_{\Rr} x m_0 dx. 
\]
Thus, 
\begin{equation}
\label{thetafor}
\Theta=-\gamma (K(0)+\bar x-\zeta).
\end{equation}
Therefore, using \eqref{lpsr}, we obtain 
\[
\varpi=-\gamma (K(0)+\bar x-\zeta)-c\bQ.
\]
Finally, we use the above results and conclude that each agent 
dynamics is
\[
\begin{cases}
\bdx=\frac{(\bar x-x)\gamma}{1+\frac{T}{c}\gamma}+\bQ\\
\bx(0)=x.  
\end{cases}
\]
In alternative, using 
\[
\bdx(t)=-\frac{\varpi+u_x(\bx(t), t)}{c}
\]
we have
\begin{equation}
\label{dyn}
\begin{cases}
\bdx(t)=\frac{(\bar \bx(t)-\bx(t))\gamma}{1+\frac{T-t}{c}\gamma}+\bQ\\
\bx(0)=x,   
\end{cases}
\end{equation}
where
\[
\bar \bx(t)=\int_{\Rr} x m(x,t)dx. 
\]
Averaging \eqref{dyn} with respect to $m$, we obtain
\begin{equation}
\label{energyid}
{\dot {\bar{\bf x}}}(t)=\bQ(t), 
\end{equation}
which is simply the conservation of energy. Thus, the trajectory of
an individual agent can be computed by combining \eqref{dyn} with \eqref{energyid} into the system
\[
\begin{cases}
\bdx(t)=\frac{(\bar \bx(t)-\bx(t))\gamma}{1+\frac{T-t}{c}\gamma}+\bQ(t)\\
{\dot {\bar{\bf x}}}(t)=\bQ. 
\end{cases}
\]
The previous system is a closed system of ordinary differential equations
that only involves $\bQ$ and the parameters of the problem. Surprisingly, it also does not depend on $\zeta$. This is due
to the fact that the average of the position of the agents is
determined by $\bQ$. Hence, the only way agents can improve their value function is by getting close to each other. This is seen in 
the mean-reverting structure in \eqref{dyn}.

\subsection{Quadratic cost with potential} 
\label{sub:quadratic_cost_with_potential}
Now, we consider a running cost with a quadratic potential. This potential penalizes the agents when the charge or stored energy deviates too much from a set point, $\kappa$. 
This penalty has the form of $\frac \eta 2(x-\kappa)^2$, where $\eta$ measures the strength of the penalty. 
Thus, we have
\[
\ell(t,x,\alpha)=c\frac{\alpha^2}{2} + \alpha\bp(t) + \frac \eta 2(x-\kappa)^2.
\]
The corresponding MFG is
\begin{align} \label{MFG_const_pot}
\begin{cases}
-u_t+\frac{(\bp(t)+u_x)^2}{2c}-\frac \eta 2(x-\kappa)^2=0\\
m_t-\frac 1 c (m(\bp(t)+u_x) )_x=0\\
\frac 1 c \int (\bp(t)+u_x) m =-\bQ(t). 
\end{cases}
\end{align}
Differentiating the Hamilton-Jacobi equation, we conclude that 
\[
-(u_x)_t+(u_x+\bp) u_{xx}-\eta (x-\kappa)=0. 
\]
We define the following quantities
\[
\Pi=\int_{\Rr} u_x m \text{ and } \Xi=\int_{\Rr} x m.
\]
Taking the time derivative on the first quantity and using the transport equation, we get
\begin{equation*}
\begin{split}
\dot \Pi &= \int_{\Rr} u_{xt}m + u_x m_t \\
&= \int_{\Rr} (\varpi+u_x)u_{xx}m - \eta \int_{\Rr} (x-\kappa) m + \int_{\Rr} u_x (m(\varpi+u_x))_x\\
&= \int_{\Rr} (\varpi+u_x)u_{xx}m - \int_{\Rr} u_{xx} m (\varpi+u_x) - \eta \int_{\Rr} (x-\kappa) m. 
\end{split}
\end{equation*}
Simplifying the preceding expression, we obtain
\begin{equation*}
\dot \Pi = -\eta (\Xi-\kappa).
\end{equation*}

Next, we take the transport equation, multiply it by $x$,  and integrate by parts,
to get
\begin{equation*}
\begin{split}
\dot \Xi&= \frac{d}{dt}\int_{\Rr} xm = \int_{\Rr} x m_t = \int_{\Rr} x \lb m(\varpi+u_x) \rb_x\\
&= -\int_{\Rr} m(\varpi+u_x) + \left[ x(m(\varpi+u_x)\right|_{\Omega}\\
&= -\varpi- \int_{\Rr} u_x m.
\end{split}
\end{equation*}
Thus, we conclude that 
\[
\dot \Xi=-\bp-\Pi.
\]
Therefore, we obtain the following averaged dynamics
\begin{equation*}
\begin{cases}
\dot \Xi = - \bp - \Pi\\
\dot \Pi = -\eta (\Xi-\kappa).
\end{cases}
\end{equation*}

Taking the time derivative of the second equation and using the first equation, we get
\begin{equation*}
\ddot \Pi - \eta \Pi = \eta(\varpi+\kappa). 
\end{equation*}
The preceding equation has the following solution
\begin{equation*}
\begin{split}
\Pi = -\kappa+e^{\sqrt{\eta}t} C_1 + e^{-\sqrt{\eta}t} C_2 
+ \frac{\sqrt{\eta}}{2}\int_0^t \lb e^{\sqrt{\eta}(t-s)}- e^{-\sqrt{\eta}(t-s)}\rb \bp(s) ds.
\end{split}
\end{equation*}
Moreover, at $t=0$, we have
\[
\dot \Pi(0)=-\eta (\Xi(0)-\kappa)=-\eta(\bar x-\kappa), 
\]
where
\[
\bar x=\int_{\Rr} x \bar m. 
\]
Thus, we need an additional constant to determine $\Pi(0)$. Given this constant, 
from the constraint equation in \eqref{MFG_const_pot}, we get
\begin{equation*}
\begin{split}
\bp_{\Pi(0)}(t) &= - \Pi - \bQ(t)\\
&=  f_{\Pi(0)}(t)	- \frac{\sqrt{\eta}}{2} \int_0^t \lb e^{\sqrt{\eta}(t-s)}- e^{-\sqrt{\eta}(t-s)}\rb \bp_{\Pi(0)}(s) ds,
\end{split}
\end{equation*}
where $f_{\Pi(0)}(t)= \kappa-e^{\sqrt{\eta}t} C_1 - e^{-\sqrt{\eta}t} C_2 -\bQ(t)$, and $C_1$ and $C_2$ are determined by the value of $\dot \Pi(0)$ and by the unknown value $\Pi(0)$.  

The preceding equation is a Volterra integral equation of the second kind with a separable kernel. In principle, 
we can solve this equation using Laplace's transform. 
The previous equation is of the form
\begin{equation} \label{eq:conv_kp}
\bp_{\Pi(0)}(t) = f_{\Pi(0)}(t) - \lambda (k*\bp_{\Pi(0)})(t),
\end{equation}
where
\[
k(t)=- \frac{\sqrt{\eta}}{2} \lb e^{\sqrt{\eta}t}- e^{-\sqrt{\eta}t}\rb
\]
and
$(k*\bp)=\int_0^t k(t-s)\bp(s) ds$ denotes the convolution product of the kernel $k$ with $\bp$.

Let $\mathcal{L}$ denote the Laplace transform. Because $\mathcal{L}\{(k*\bp)(t)\}=\mathcal{L}\{k(t)\}\mathcal{L}\{\bp(t)\}$, applying the Laplace transform to \eqref{eq:conv_kp} yields
\begin{equation*}
\mathcal{L}\{\bp(t)\} = \mathcal{L}\{f_{\Pi(0)}(t)\} + \lambda \mathcal{L}\{k(t)\}\mathcal{L}\{\bp(t)\}.
\end{equation*}
Simplifying the above equation, we obtain
\begin{equation*}
\bp_{\Pi(0)}(t)= \mathcal{L}^{-1}\left\{\frac{\mathcal{L}\{f_{\Pi(0)}(t)\}}{1-\lambda \mathcal{L}\{k(t)\}}\right\},
\end{equation*}
where $\mathcal{L}^{-1}$ is the inverse Laplace transform.


Finally, we take the resulting expression for $\varpi_{\Pi(0)}$ into 
the Hamilton-Jacobi equation, solve it and obtain a function
$u_{\Pi(0)}(x,t)$. Then, the value $\Pi_0$ is determined implicitly
by the equation
\begin{equation}
\label{imppi}
{\Pi(0)}=\int_{\Rr} (u_{\Pi(0)})_x(x,0)m_0 dx. 
\end{equation}

In the case of quadratic terminal data, 
\[
u(x,T)= \frac{\gamma}{2}(x-\zeta)^2, 
\]
we can reduce the solution of the Hamilton-Jacobi equation
into solving ordinary differential equations. 
For that, we look for a solution
\[
u(x,t)=\theta_0(t)+\theta_1(t) x+ \theta_2(t) x^2 
\]
satisfying 
\[
u(x,T)= \frac{\gamma}{2}(x-\zeta)^2=\theta_0(T)+\theta_1(T) x+ \theta_2(T) x^2. 
\]
Then, the first equation in \eqref{MFG_const_pot} becomes
\[
-(\dot \theta_0+\dot \theta_1(t) x+ \dot \theta_2(t) x^2)+\frac{(\bp(t)+\theta_1(t) + 2\theta_2(t) x)^2}{2c}-\frac \eta 2(x-\kappa)^2=0. 
\]
Thus, by matching powers of $x$, we obtain differential equations for
$\theta_i$, $0\leq i\leq 2$. The resulting expression can be used in 
\eqref{imppi} to obtain the solution. 

%
%
%
%

\section{Real Data} 
\label{sec:real_data}
	In this section, we use real data of daily energy consumption in the UK, during a twenty-four hour period.
	The data is available online at
 \url{https://www.nationalgrid.com/uk/}. In Figure \ref{fig:Q}, we plot the    power 
	supply oscillation $\bQ$ (which is simply the negative of the demand) normalized to have mean zero over 24 hours.
	\begin{figure}[h]
		\centering
		\includegraphics[width=.7\textwidth]{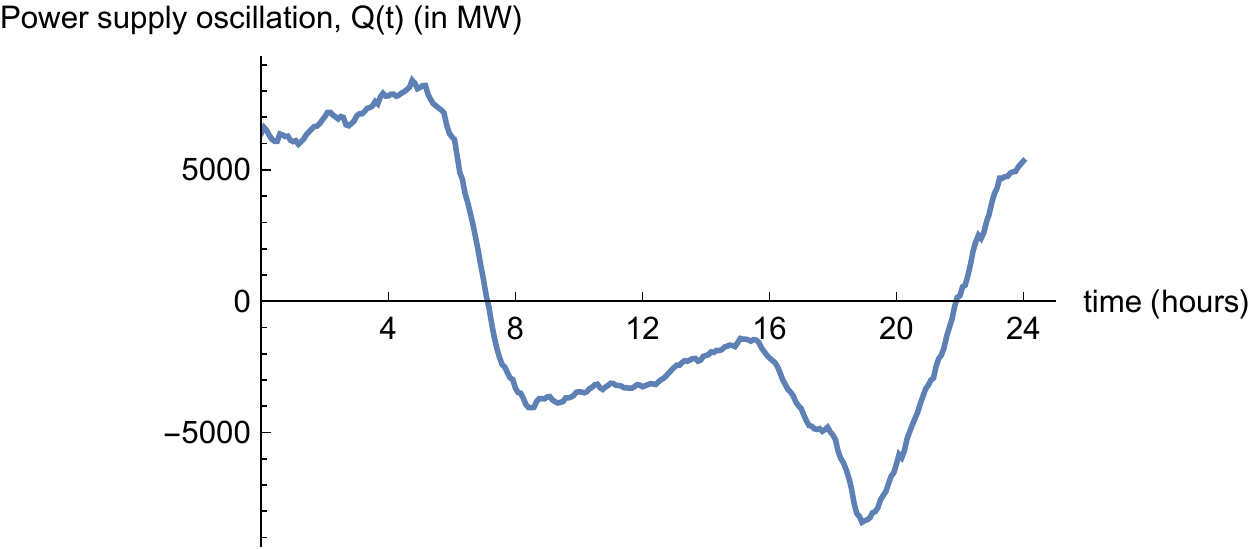}
		\caption{Normalized electricity production $\bQ$.}
		\label{fig:Q}
	\end{figure}

	We compare our price-formation model with the MFG model presented in \cite{de2016distributed}. 
	In that model, the  energy price is a function of the aggregate consumption. In that case, 
	the price is not determined by a supply vs demand condition and, thus, there may 
	be an energy imbalance. 
	Here, we consider the state-independent quadratic cost model from Section \ref{sub:state_independent_quadratic_cost}.  In our model, the price depends only on the constant that accounts for battery's wear and tear.  
	This constant can be empirically estimated, but, here, we calibrate our model against the model in  \cite{de2016distributed} using a least squares approach. 
	Let $\vartheta$ be the priced computed in \cite{de2016distributed}. According to \eqref{lpsr}, the price given is 
	$\varpi^{c, \Theta} =\Theta-c \bQ$. Thus, 
	we estimate the value of $c$, by	
	solving the minimization problem
	\begin{equation}
		\min_{c, \Theta\in \Rr} \| \varpi^{c, \Theta} - \vartheta \|_2^2 = \min_{c, \Theta\in \Rr} \| \Theta - c\bQ - \vartheta \|_2^2,  
	\end{equation}
    and, using $N=10^6$ agents, we obtained $c=0.00172  \$ (\text{kW})^{-2} h^{-1}$.
	
	The price given by our model is plotted in Figure \ref{fig:price}. We predict smaller
	peak oscillations and thus, our methods may help stabilize the market. 
	\begin{figure}[h]
	  \centering
	    \includegraphics[width=.7\textwidth]{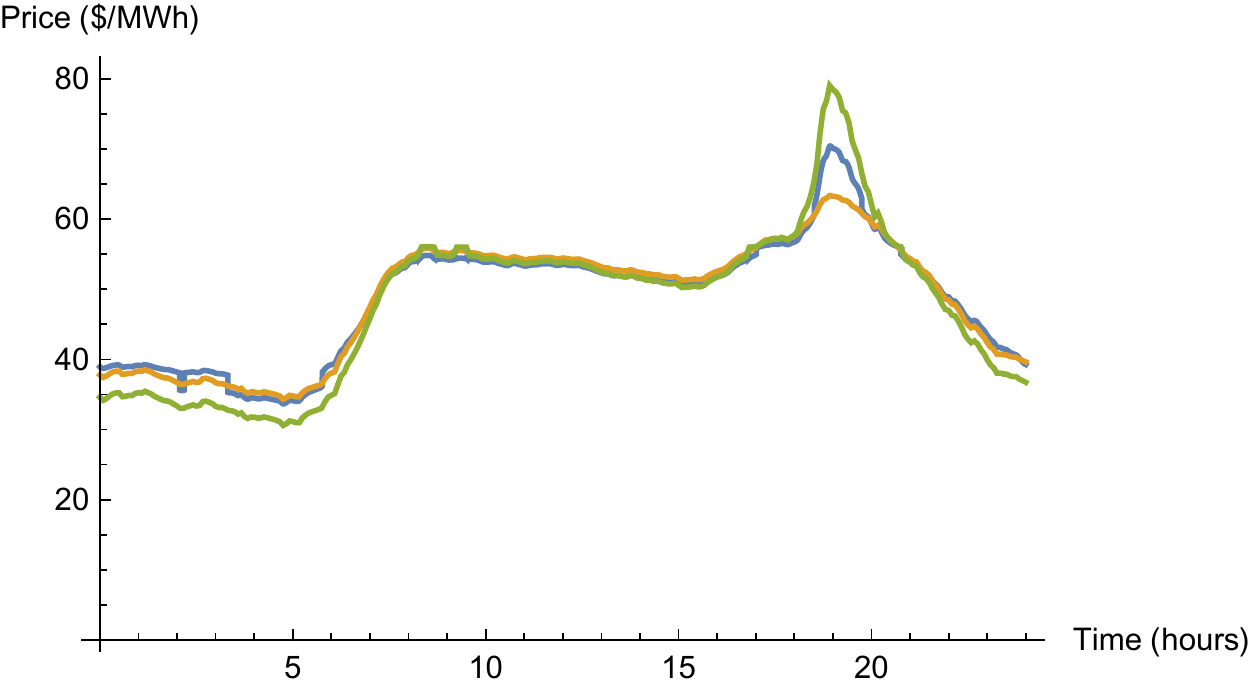}
	  \caption{Evolution of electricity price during a twenty-four hour period. In green, we plot the energy's price when no batteries are connected to the grid. In blue, we plot the price with batteries connected to the grid and the price is given by the model in \cite{de2016distributed}. In yellow, we plot the price corresponding to our model.}
	  \label{fig:price}
	\end{figure}
	

\section{Conclusions and extensions}

Here, we described a model for price formation in electricity markets, proved 
the well-posedness of the problem, and developed methods to compute the solutions. Our model has a minimal number of features and fits well real data. 
In addition, our model may have stabilizing properties of the price 
at peak consumption. 

Several
extensions of our model are of interest. First, we can consider the case
where the supply $\bQ(\varpi, t)$ depends on price. Provided the supply increases with the price, which is a natural assumption from the economic point of view, the solvability conditions are similar. 
In particular, \eqref{pode} becomes
\[
\dot \varpi \left[\frac{\partial \bQ}{\partial \varpi}+\int_{\Rr} D^2_{pp} H m dx\right]
=-\frac{\partial \bQ}{\partial t}-\int_{\Rr} \left[D^2_{pp} H D_xH m+\epsilon D^{3}_{ppp}H u_{xx}^2 m\right]dx. 
\]
Thus, we obtain similar bounds for $\dot \varpi$ if $\frac{\partial \bQ}{\partial \varpi}\geq 0$.
Therefore, the existence theory follows a similar argument. 
Moreover,  if $\frac{\partial \bQ}{\partial \varpi}\geq 0$,   
the operator $A$ in section \eqref{suniq} is monotone and, therefore,  uniqueness
of solution holds. 

In real applications, $\bQ$ may depend on delayed prices. While this does not fit directly into our framework, we can consider a Taylor expansion:
\begin{align*}
\bQ(\varpi(t-\tau), t)\simeq &
\bQ(\varpi(t), t)-\tau \frac{\partial\bQ(\varpi(t), t)}{\partial\varpi}\dot \varpi(t)\\&
+\frac{\tau^2}{2}\left[
\frac{\partial\bQ(\varpi(t), t)}{\partial\varpi}\ddot \varpi + 
\frac{\partial^2 \bQ(\varpi(t), t)}{\partial \varpi^2}\dot \varpi^2(t)
\right]+\hdots
\end{align*}
Thus, it is natural to look at the case where $\bQ$ depends on the 
price and its derivatives. 

Finally, a natural extension is the case where $\bQ$ has random fluctuations. This is particularly relevant if the energy production is subject to unpredictable changes - this is the case of wind energy. For the case where $\bQ$ is random, we need to use the master equation as in \cite{LCDF}.

%
%
%

\bibliographystyle{plain}
\def\cprime{$'$}

\def\cprime{$'$}

\end{document}